\documentclass[11pt]{amsart}

\usepackage[active]{srcltx}
\usepackage{amsfonts}
\usepackage{amssymb}

\date{}
\begin{document}
\def \Z{\Bbb Z}
\def \C{\Bbb C}
\def \R{\Bbb R}
\def \Q{\Bbb Q}
\def \N{\Bbb N}
\def \tr{{\rm tr}}
\def \span{{\rm span}}
\def \Res{{\rm Res}}
\def \End{{\rm End}}
\def \E{{\rm End}}
\def \Ind {{\rm Ind}}
\def \Irr {{\rm Irr}}
\def \Aut{{\rm Aut}}
\def \Hom{{\rm Hom}}
\def \mod{{\rm mod}}
\def \ann{{\rm Ann}}
\def \<{\langle}
\def \>{\rangle}
\def \t{\tau }
\def \a{\alpha }
\def \e{\epsilon }
\def \l{\lambda }
\def \L{\Lambda }

\def \b{\beta }
\def \om{\omega }
\def \o{\omega }
\def \c{\chi}
\def \ch{\chi}
\def \cg{\chi_g}
\def \ag{\alpha_g}
\def \ah{\alpha_h}
\def \ph{\psi_h}
\newcommand{\bea}{\begin{eqnarray}}
\newcommand{\eea}{\end{eqnarray}}
\newcommand{\Zp}{{\mathbb Z} _{> 0} }

\def \be{\begin{equation}\label}
\def \ee{\end{equation}}
\def \bl{\begin{lem}\label}
\def \el{\end{lem}}
\def \bt{\begin{thm}\label}
\def \et{\end{thm}}
\def \bp{\begin{prop}\label}
\def \ep{\end{prop}}
\def \br{\begin{rem}\label}
\def \er{\end{rem}}
\def \bc{\begin{coro}\label}
\def \ec{\end{coro}}
\def \bd{\begin{de}\label}
\def \ed{\end{de}}
\def \pf{{\bf Proof. }}
\def \voa{{vertex operator algebra}}

\newtheorem{thm}{Theorem}[section]
\newtheorem{prop}[thm]{Proposition}
\newtheorem{coro}[thm]{Corollary}
\newtheorem{conj}[thm]{Conjecture}
\newtheorem{lem}[thm]{Lemma}
\newtheorem{rem}[thm]{Remark}
\newtheorem{example}[thm]{Example}
\newtheorem{de}[thm]{Definition}
\newtheorem{hy}[thm]{Hypothesis}
\makeatletter \@addtoreset{equation}{section}
\def\theequation{\thesection.\arabic{equation}}
\makeatother \makeatletter

\newcommand{\g}{\frak g}
    \newcommand{\nno}{\nonumber}
    \newcommand{\lbar}{\bigg\vert}
    \newcommand{\p}{\partial}
    \newcommand{\dps}{\displaystyle}
    \newcommand{\bra}{\langle}
    \newcommand{\ket}{\rangle}
 \newcommand{\res}{\mbox{\rm Res}}
\renewcommand{\hom}{\mbox{\rm Hom}}
  \newcommand{\epf}{\hspace{2em}$\Box$}
 \newcommand{\epfv}{\hspace{1em}$\Box$\vspace{1em}}
\newcommand{\nord}{\mbox{\scriptsize ${\circ\atop\circ}$}}
\newcommand{\wt}{\mbox{\rm wt}\ }

\title[]{Some general results on conformal embeddings of affine vertex operator algebras}

\author{Dra\v{z}en  Adamovi\' c and Ozren Per\v{s}e }
\address{Faculty of Science - Department of Mathematics, University of Zagreb, Croatia}
\email{adamovic@math.hr; perse@math.hr} \maketitle

\begin{abstract}
We give a general criterion for conformal embeddings of
vertex operator algebras associated to affine Lie algebras at arbitrary levels. Using that criterion, we construct new
conformal embeddings at admissible rational and negative integer
levels. In particular, we construct all remaining conformal
embeddings associated to automorphisms of Dynkin diagrams of simple
Lie algebras. The semisimplicity of the corresponding decompositions
is obtained by  using the concept of  fusion rules for vertex
operator algebras.
\end{abstract}

\section{Introduction}

Let $U$ and $V$ be vertex operator algebras of affine type. We say that $U$ is conformally  embedded into $V$ if $U$ can be
realized as a vertex subalgebra of $V$ with the same Sugawara Virasoro vector.
 The most interesting case of conformal embedding is when $V$ is a finite direct sum of irreducible $U$--modules.
 These conformal embeddings are studied in various  aspects of conformal field theory
 (cf. \cite{AGO},
\cite{BB}, \cite{BEK}, \cite{EG},\cite{Fuchs}, \cite{SW}), the theory of tensor categories (cf. \cite{FFRS}, \cite{KO})
and in the representation theory of affine Kac-Moody Lie algebras (cf. \cite{CKPP}).
The  construction and classification of conformal embeddings have mostly been studied for simple affine vertex operator algebras of positive levels.
 Some examples of conformal embeddings at rational admissible levels were constructed in \cite{P2} and \cite{P4}.

 In the present paper
 we give both necessary and sufficient conditions for
conformal embeddings at general levels, within the framework of
vertex operator algebra theory. In this way we will be able to construct new examples of conformal embeddings at rational and negative integer levels. Let us explain our results in more
details.

Let $\g$ be a simple finite-dimensional Lie algebra and $\g_0$ its
subalgebra which is a reductive Lie algebra. Let $N_{\g}(k,0)$ be
the universal affine vertex operator algebra of level $k \neq
-h^{\vee}$ associated to $\g$, and $L_{\g}(k,0)$ its simple
quotient. Then $\g_0$ generates a subalgebra $U_{\g_0}$ (resp.
$\widetilde{U}_{\g_0}$) of $N_{\g}(k,0)$ (resp. $L_{\g}(k,0)$). Let
$\omega_{\g_0}$ be the Virasoro vector in $U_{\g_0}$ which is a sum
of usual Virasoro vectors obtained by using Sugawara construction.
Let $ \overline{L}(z) = \sum_{n \in \Z } \overline{L}(n) z ^{-n-2}$
be the associated Virasoro field. We prove the following result:

\begin{thm} \label{thm-intro}
Assume that $\g = \g_0 \oplus W$  such that $W$ is a $\g_0$--module orthogonal to $\g_0$.
Then $\widetilde{U}_{\g_0}$ is conformally embedded into
$L_{\g}(k,0)$ if and only if the following condition holds
\bea \overline{L}(0) x(-1){\bf 1} = x(-1) {\bf 1} \qquad (x \in W).
\label{eqn-intro}\eea
  Moreover, the condition (\ref{eqn-intro}) implies that $N_{\g}(k,0)$ contains a singular vector of
conformal weight $2$.
\end{thm}

If $k$ is a positive integer, then  $N_{\g}(k,0)$ has a singular
vector of conformal weight $2$ if and only if $k=1$. So we only have
conformal  embeddings into $L_{\g}(1,0)$. But in general we have
such singular vectors for other values of $k$ (cf. Remark
\ref{sing-napomena}).

Next we apply Theorem  \ref{thm-intro} in the case when $\g_0$ is a simple Lie algebra,
 and
$$\g =\g_0 \oplus V_{\g_0}(\mu_1) \oplus \cdots \oplus V_{\g_0}(\mu_s),$$
where $V_{\g_0}(\mu_i)$ is irreducible finite-dimensional
$\g_0$-module with highest weight $\mu_i$, and $V_{\g_0}(\mu_i)$
($i=1, \ldots ,s$) is orthogonal to $\g_0$ with respect to the
invariant bilinear form on $\g$. Denote by $ \widetilde{L}_{\g_0}
(k' , 0)$ the subalgebra of $L_{\g}(k,0)$ generated by $\g_0$, where
$k' = a k$ and $a \in \N$  is the Dynkin index of  the embedding $\g_0 < \g$. Note
that in the case of general levels, vertex algebra $
\widetilde{L}_{\g_0}  (k' , 0)$ is not necessarily simple.

Theorem \ref{thm-intro} then implies that the key condition for the
conformal embedding is that the eigenvalue of the Casimir operator
of $\g_0$ is the same for all $V_{\g_0}(\mu_i)$, $i=1, \ldots ,s$.
More precisely, if
\bea && ( \mu_ i,  \mu_ i + 2  \rho _0 )_0   =  (\mu_j, \mu_j + 2
\rho_0 )_0 \qquad \mbox{ for all }  i,j=1, \ldots ,s, \nonumber \eea
where $( \cdot, \cdot)_0$ denotes the suitably normalized invariant
bilinear form on $\g_0$, then for $k,k'$ such that
\bea && \frac{ ( \mu_i , \mu _i + 2 \rho_0 ) _0}{ 2 (k ' + h _0
^{\vee} )} = 1, \quad (i=1, \ldots, s), \nonumber \eea
vertex operator algebra $ \widetilde{L}_{\g_0}  (k' , 0)$ is
conformally embedded into $L_{\g}(k , 0)$.

In order to obtain the conformal embedding of simple vertex operator
algebras $L_{\g_0}(k',0) < L_{\g}(k,0)$ and complete reducibility of
$L_{\g}(k,0)$ as a $L_{\g_0}(k',0)$--module, we additionally assume
that $\g_0$ is obtained as a fixed point subalgebra of an
automorphism of $\g$ of finite order. Then this automorphism
naturally acts on $L_{\g} (k,0)$, so one can use the results  from orbifold theory \cite{DM}. Under certain natural conditions we get
$$ L_{\g}(k,0) = L_{\g_0}(k',0) \oplus L_{\g_0}(k',\mu_1) \oplus \cdots \oplus L_{\g_0}(k',\mu_s).$$
These conditions involve some information about fusion rules of
$L_{\g_0}(k',0)$--modules (cf. Theorem \ref{general}).

Our methods enable us to construct all conformal embeddings
associated to automorphisms of Dynkin diagrams. Let us illustrate
this result by the following table:
\begin{enumerate}
\item[{}]
 Table A, \quad ${\g}$, $\g_0$ are  simple Lie
algebras, $\g_0 < \g$.
$$
\begin{array}{| c | c | c|  c | c | }
\hline
{\g} & {\g}_0& \mbox{decomposition of} \ L_{\g}(k,0)& k&
k'    \\ \hline
A_{2\ell-1} & C_{\ell} & L_{\g_0}(k',0) \oplus L_{\g_0} ( k',\omega_2) & -1 & -1 \\
\hline
A_{2\ell}  & B_{\ell} & L_{\g_0}(k',0) \oplus L_{\g_0} (k', 2 \omega_1) & 1 & 2 \\
\hline
A_{2}  & A_{1} & L_{\g_0}(k',0) \oplus L_{\g_0} (k', 4 \omega_1) & 1 & 4 \\
\hline
D_{\ell} & B_{\ell -1} & L_{\g_0}(k',0)  \oplus L_{\g_0} (k',\omega_1) & -\ell + 2 & -\ell + 2\\
\hline
E_{6} & F_{4} & L_{\g_0}(k',0)  \oplus L_{\g_0} (k',\omega_4) & -3 & -3  \\
\hline
D_{4} & G_{2} & L_{\g_0}(k',0) \oplus L_{\g_0} (k',\omega_2) \oplus & -2 & -2 \\
 &  \ & \ L^{(1)} _{\g_0}(k',\omega_1) \oplus L^{(2)} _{\g_0}(k',\omega_1) & \  & \\
\hline
\end{array}
$$
\end{enumerate}

In order to describe all conformal embeddings  for such
automorphisms, we include into this table some well-known conformal
embeddings at positive integer levels in the cases $(A_{2 \ell},
B_{\ell})$, $(A_2, A_1)$ (cf. \cite{Fuchs}; see also \cite{AP-2010},
\cite{W}) and the conformal embedding for pair $(A_{2\ell-1},
C_{\ell})$ from \cite{AP}. But other conformal embeddings at
negative integer levels are new.

We also apply our methods to affine vertex operator algebras at
admissible levels (cf. \cite{A-1994}, \cite{AM}, \cite{KW1},
\cite{KW2}, \cite{P1}, \cite{P2}, \cite{P3}, \cite{W}) and get new
conformal embedding of $L_{A_2}(-5/3,0)$ into $L_{G_2}(-5/3,0)$ with
decomposition
$$L_{G_2}( -5/3,0) = L_{A_2}(-5/3,0) \oplus L_{A_2}(-5/3,\omega_1) \oplus L_{A_2}(-5/3, \omega_2).$$
We also determine the decomposition for the conformal embedding of
$L_{D_{\ell}} (-\ell + 3/2,0)$ into $L_{B_{\ell}}(-\ell+3/2,0)$ from
\cite{P4}.

We should also mention that our methods give an uniform proof for a
family of conformal embeddings which does not require explicit
realizations of affine Lie algebras nor explicit formulas for
singular vectors. In most cases one can check all conditions for
conformal embeddings only by using tensor product decompositions of
finite-dimensional modules of simple Lie algebras.

The notion of conformal embedding is also closely related to the
notion of extension of vertex operator algebra. We note that simple
current extensions of affine vertex operator algebras at positive
integer levels were studied in \cite{DLM}, \cite{L-sc}.

We assume that the reader is familiar with the notion of vertex
operator algebra (cf. \cite{B}, \cite{FHL}, \cite{FLM}).

\section{ A general criterion for conformal embeddings}
\label{kriterij}

In this section we shall derive a general criterion for the
conformal embedding of affine vertex operator algebras associated to
a pair of Lie algebras $(\g, \g_0)$, where $\g$ is a simple Lie
algebra and $\g_0$ its reductive subalgebra.

Let $\g$ be a simple complex finite-dimensional Lie algebra with a
non-degenerate invariant symmetric bilinear form $(\cdot, \cdot)$.
We can normalize the form on $\g$ such that $(\theta, \theta) = 2$,
where $\theta$ is the highest root of $\g$. Let  $\rho$ be the sum
of all fundamental weights for $\g$  and let $h ^{\vee}$ be the dual
Coxeter number for $\g$.

Assume that  $\g_0$ is a Lie subalgebra of
$\g$ such that
\begin{itemize}

\item[(1)] bilinear form $( \cdot, \cdot)$ on $\g_0$ is   non-degenerate  and
$\g = \g_0 \oplus W$ is an orthogonal sum of $\g_0$ and  a
$\g_0$--module $W$;

\item[(2)] $\g_0 = \oplus_{i=0} ^n  \g_{0,i}$ is an orthogonal sum of commutative subalgebra $\g_{0,0}$ and simple Lie algebras
$\g_{0,1}, \ldots , \g_{0,n}$ (i.e., $\g_0$ is a reductive Lie
algebra).
\end{itemize}

Let  $(\cdot, \cdot)_{\g_{0,i}}$ be the non-degenerate bilinear form
on $\g_{0,i}$ normalized such that $(\theta_{0,i}, \theta_{0,i}
)_{\g_{0,i}} =2$, where $\theta_{0,i}$ is the highest root of
$\g_{0,i}$.  One can show that
$$ a_i \ (\cdot, \cdot)_{\g_{0,i} } =  (\cdot, \cdot) \vert \g_{0,i} \times \g_{0,i}\quad  \mbox{for certain} \ a_i \in \N. $$
(For $i=0$, we can take $a_0 =1$).

The $(n+1)$--tuple $(a_0,a_1, \cdots, a_n)$ is called the Dynkin multi-index of embedding $\g_0$ into $\g$.
Let $\{ u ^{i,j} \}_j$, $\{ v ^{i,j} \}_j$ be a pair of dual bases
for  $\g_{0,i}$ such that $ ( u ^{i,j}, v ^{i,k} )_{\g_{0,i}} =
\delta_{j,k}$. Let  $h_{0,i} ^{\vee}$  be the dual Coxeter number
for  $\g_{0,i}$.  (For $i = 0$ we take $h_{0,0} ^{\vee} = 0$).

Let $N_{\g} (k ,0)$ be the universal affine vertex operator algebra
of level $k \neq -h^{\vee}$ associated to $\g$ (cf. \cite{FrB},
\cite{FZ}, \cite{K2}, \cite{LL}, \cite{L}). Let $N_{\g}  ^{1} (k ,
0)$ be the maximal ideal  of   $N_{\g} (k , 0)$, and  $L_{\g} (k ,
0) =  N_{\g} (k , 0) / N_{\g} ^1 (k , 0)$ the corresponding simple
vertex operator algebra.

Then $\g_{0,i}$ generates a vertex subalgebra of $N_{\g}(k,0)$  which is isomorphic to $N_{\g_{0,i}}(k_i,0)$ for $k_i = a_i k$.
Moreover, $\g_0$ generates a subalgebra of $N_{\g}(k,0)$ which is
isomorphic to
$$  U_{\g_0}= \otimes_{i=0} ^n N_{\g_{0,i}} (k_i,0).$$

Let $\widetilde{U}_{\g_0}$ be the subalgebra of $L_{\g}(k,0)$ generated by $\g_0$. Then $\widetilde{U}_{\g_0}$ is isomorphic to
$$ \widetilde{U}_{\g_0}= \otimes_{i=0} ^n \widetilde{L}_{\g_{0,i}} (k_i,0),$$
where $\widetilde{L}_{\g_{0,i}}(k_i,0)$ is a certain quotient of $N_{\g_{0,i}}(k_i,0)$.

\begin{de}
We say that the vertex operator algebra $\widetilde{U}_{\g_0}$ is
conformally embedded into $L_{\g}(k,0)$, if $\widetilde{U}_{\g_0}$
is a vertex subalgebra of $L_{\g}(k,0)$ with the same Virasoro
vector.
\end{de}

Let $\omega_{\g}$ be the usual  Sugawara Virasoro vector in $N_{\g}(k,0)$ and let
$$\omega_{\g_{0,i}} = \frac{1}{2 (k_i + h_{0,i} ^{\vee} )}\sum_{j = 1} ^{\dim \g_{0,i} } u^{i,j} (-1) v ^{i,j} (-1) {\bf 1} $$ be the Sugawara Virasoro vector in $N_{\g_{0,i} }(k_i, 0)$ $(i=0, \dots, n)$.

Then
$$\omega_{\g_0} = \omega_{\g_{0,0}} + \cdots + \omega_{\g_{0,n}}$$
is the Virasoro vector in $U_{\g_0}$.

As usual we identify $x \in \g$ with $x(-1) {\bf 1} \in
N_{\g}(k,0)$. Therefore, $W$ can be considered as a subspace of
$N_{\g}(k,0)$.

Let
$$\overline{L} (z) = \sum_{n \in \Z} \overline{L}(n) z ^{-n-2}= Y(\omega_{\g_0},z). $$

\begin{thm} \label{konf-1}
Assume that
\bea \label{cond-1} && \overline{L}(0) v = v \quad \mbox{for every } \ v \in W. \eea
Then
$$ \omega_{\g} - \omega_{\g_0} \in N^1 _{\g} (k,0). $$
In particular, $\widetilde{U}_{\g_0}$ is conformally embedded into $L_{\g}(k,0)$ if and only if the condition (\ref{cond-1}) holds.
\end{thm}
\begin{proof}
It suffices to prove that
\bea   \label{singular-condition}x(1) ( \omega_{\g} - \omega_{\g_0} ) = 0 \qquad (x \in \g). \eea
By using properties  of the Virasoro vector $\omega_{\g}$ we get
\bea &&  x(1) \omega_{\g} = x(-1){\bf 1} \quad (x \in \g),   \\
&&    x(1) \omega_{\g_0} = x(-1){\bf 1} \quad (x \in \g_0).   \eea
Therefore the condition (\ref{singular-condition}) holds for $ x \in \g_0$.
Assume now that $x \in W$. Since $W \bot \g_0$, we conclude that
\bea  x(1) \omega_{\g_0}  &=&\sum_{i = 0} ^{ n } \sum_{j = 1} ^{\dim \g_{0,i} }\frac{1}{2 (k_i + h_{0,i} ^{\vee} )} [[x, u^{i,j}],  v ^{i,j}] (-1) {\bf 1}
\nonumber \\
& = & \overline{L} (0) x (-1) {\bf 1} = x(-1) {\bf 1} .
\eea
This proves that  $\widetilde{U}_{\g_0}$ is conformally embedded into $L_{\g}(k,0)$.

It is also clear that the condition (\ref{cond-1}) is a necessary condition for conformal embedding.
\end{proof}

\begin{rem}
If $\widetilde{U}_{\g_0}$ is conformally embedded into
$L_{\g}(k,0)$, the equality of Virasoro vectors implies the equality
of central charges. By using usual formulas for central charges of
Sugawara Virasoro vectors we get
\bea  && \dim (\g_{0,0}) + \frac{k _1 \dim (\g_{0,1})}{ k_1 +
h_{0,1} ^{\vee} } + \cdots +  \frac{k _n \dim (\g_{0,n})}{ k_n +
h_{0,n} ^{\vee} } = \frac{k  \dim (\g)}{ k + h ^{\vee} } .
\label{num-cond} \eea

So our condition (\ref{cond-1}) from Theorem \ref{konf-1} implies numerical criterion (\ref{num-cond}).
But, in general, (\ref{num-cond}) is only a necessary condition for conformal embeddings.
\end{rem}
\begin{rem} \label{sing-napomena}
Assume that the  conditions from  Theorem \ref{konf-1}  are
satisfied. Then $N_{\g}(k,0)$ must have at least one singular vector
at conformal weight $2$. Some explicit formulas for these singular
vectors appeared in \cite{A-1994}, \cite{A-2003}, \cite{P1},
\cite{P2}, \cite{P3} and recently in \cite{AK-2010}.
\end{rem}

We shall now see how the theory presented here fits into some   examples.

\begin{example} \label{prvi}
Let $\g$ be a simple complex Lie algebra of type $A$, $D$ or $E$, ${\frak h}$ its Cartan subalgebra and $\Delta$ its root system.
Then
$$\g = {\frak h} \oplus W, \qquad W = \bigoplus_{ \alpha \in \Delta} \g_{\alpha}.$$
Take now $\g_0 = {\frak h}$ in the above construction. Then
$$ \overline{L}(0) \vert W  \equiv \mbox{Id} \quad \mbox{iff} \quad k=1. $$
This leads to the conformal embedding
$$L_{\g_0} (1,0) < L_{\g}(1,0). $$
Of course, this can be easily verified  by using explicit vertex operator construction of  $L_{\g}(1,0)$ (cf. \cite{FLM}, \cite{K2}).
\end{example}

\begin{example}
Let $\g = \frak{sp}(2 \ell, \C)$. Then $\g_0 = \frak{gl}(\ell, \C)$ can be realized as a subalgebra of $\g$ spanned by Cartan subalgebra ${\frak h}$ and short root vectors.
We have
$$ \g = \g_0 \oplus W, \qquad W= V_{\g_0} (2 \omega_1) \oplus V_{\g_0}(2 \omega_{\ell} ) . $$
Then
$$ \overline{L}(0) \vert W  \equiv \mbox{Id} \quad \mbox{iff} \quad k=1 \ \mbox{or} \ k=-1/2. $$
This leads to the conformal embeddings
$$L_{\g_0} (2,0) < L_{\g}(1,0), \qquad \widetilde{L}_{\g_0} (-1, 0) < L_{\g} (-1/2,0). $$
\end{example}

\begin{example}
 By using the same method as above we obtain the following conformal embeddings:

\item[(1)] Let $\g = \frak{sl} (\ell +1, \C)$, $\g_0 = \frak{gl}(\ell, \C)$. Then we have conformal embeddings:
$$L_{\g_0} (1,0) < L_{\g}(1,0), \qquad \widetilde{L}_{\g_0} (- \tfrac{\ell+1}{2}, 0) < L_{\g} (-\tfrac{\ell+1}{2},0). $$
\item[(2)] Let $\g = \frak{o} ( 2 \ell, \C)$, $\g_0 = \frak{gl}(\ell, \C)$. Then we have conformal embeddings:
$$L_{\g_0} (1,0) < L_{\g}(1,0), \qquad \widetilde{L}_{\g_0} (- 2, 0) < L_{\g} (- 2,0). $$

\end{example}

\begin{rem}   One can show that in examples presented above  the vertex operator algebra  $L_{\g}(k,0)$ is not a finite sum of irreducible $L_{\g_0}(k',0)$--modules.
But it is still possible that there exists certain twisted
$L_{\g}(k,0)$--module which is a finite sum of  irreducible twisted
$L_{\g_0}(k',0)$--modules. In particular, this holds for the
$\Z_2$--twisted module for $L_{\g}(1,0)$ which gives principal
realization of level one modules for affine Lie algebras in Example
\ref{prvi}.
\end{rem}

It seems that to get conformal embeddings such that  $L_{\g}(k,0)$ is a finite  sum of irreducible $L_{\g_0}(k',0)$--modules,
 one needs to consider the case when $\g_0$ is semisimple. In this article we shall make the first step and consider the important case when $\g_0$ is a simple Lie algebra.

 So we shall now assume that $\g_0$ is a simple Lie algebra. In the above settings we take $n=1$, $\g_0 = \g_{0,1}$.
 Set $$(\cdot, \cdot)_0 =(\cdot, \cdot) _{\g_{0,1}}, \quad h_0 ^{\vee} = h_{0,1} ^{\vee}, \quad a= a_1. $$   Let  $\rho_0$  be  the sum of
 all dominant integral weights for $\g_0$.

Assume that
$$
\g = \g_0 \oplus \g_1 \oplus \cdots \oplus \g_s,$$
such that
\begin{itemize}
\item[(1)]$\g_i = V_{\g_0}(\mu_i)$ is an irreducible finite-dimensional highest weight $\g_0$-module with highest weight $\mu_i$,

\item[(2)] $\g_i \bot \g_0$ for  $ i =1, \dots, s$ (with respect to $(\cdot, \cdot)$),
 \item[(3)] $   ( \mu_ i,  \mu_ i + 2  \rho _0 )_0   =  (\mu_j, \mu_j + 2 \rho_0 )_0 \qquad   (i,j > 0). $
\end{itemize}

  Then $\g_0$ generates
a subalgebra of $N_{\g} (k , 0)$ isomorphic to $ U_{\g_0} \cong
N_{\g_0} (k ' ,0)$, where $k ' = a k $.

Let $\omega_{\g}$ (resp. $\omega_{\g_0}$) be the Virasoro vector in $N_{\g}(k , 0)$
(resp. $N_{\g_0} (k' , 0)$) obtained by the Sugawara construction.

In what follows we choose $k$ such that
\bea && \frac{ ( \mu_i , \mu _i + 2 \rho_0 ) _0}{ 2 (k ' + h _0 ^{\vee} )} = 1, \quad (i=1, \ldots, s). \label{uv-k} \eea

Let $W = \g_1 \oplus \cdots \oplus \g_s$. Then $W$ is a $\g_0$--module such that
$$ \overline{L}(0) x(-1) {\bf 1} = x(-1) {\bf 1} \qquad (x \in W).$$
As before, let $\widetilde{U}_{\g_0} = \widetilde{L}_{\g_0}  (k' ,
0)$ be the subalgebra of $L_{\g}(k , 0)$ generated by $\g_0$. Now
Theorem \ref{konf-1} implies the following result:

 \begin{coro} \label{konf-ul} Assume that the above conditions hold.
 Then $ \widetilde{L}_{\g_0}  (k' , 0)$ is conformally embedded into  $L_{\g}(k , 0)$.
 \end{coro}

The simplicity of the vertex operator algebra $ \widetilde{L}_{\g_0} (k'
, 0)$ will be  investigated in the next section.

\section{On semisimplicity of $ L_{\g}  (k , 0)$ as a $\widehat{\g_0}$--module}
\label{kriterij-poluprost}

In this section we give sufficient conditions for complete
reducibility of $ L_{\g}  (k , 0)$ as a $\widehat{\g_0}$--module. In
particular, we obtain the conformal embedding of simple vertex
operator algebras $L_{\g_0}(k',0) < L_{\g}(k,0)$.

We assume that $\g_0$ is a simple Lie algebra and that  all conditions on embedding $\g_0 < \g$ from Section
\ref{kriterij} hold. Furthermore, we assume that $\g_0$ is the fixed
point subalgebra of an automorphism $\sigma$ of $\g$ of order $s+1$,
and that $V_{\g_0}(\mu_i)$ is the eigenspace associated to the
eigenvalue $\xi ^i$ (for $i=1, \ldots ,s$), where $\xi$ denotes the
corresponding primitive root of unity. Then $\sigma$ can be extended
to a finite-order automorphism of the simple vertex operator algebra
$L_{\g} (k,0)$ which admits the following decomposition
$$L_{\g}(k,0) = L_{\g}(k,0) ^0 \oplus L_{g}(k,0) ^1 \oplus \cdots \oplus L_{\g}(k,0) ^s$$
where
$$L_{\g}(k,0) ^{i} = \{ v \in L_{\g}(k,0) \ \vert \ \sigma (v) = \xi ^i v \}. $$
Clearly $L_{\g}(k,0) ^{i}$ is a $\widehat{\g_0}$--module.

For a dominant weight $\mu$ for $\g_0$ we define $\widehat{\mu} = k'
\Lambda_0 + \mu$. Denote by $L_{\g_0} (k', \mu)$ the irreducible
highest weight $\widehat{\g_0}$--module with highest
weight~$\widehat{\mu}$. Note that the lowest conformal weight of any
$\widehat{\g_0}$--module with highest weight~$\widehat{\mu}$ is
given by the formula
\bea && \frac{ ( \mu , \mu + 2 \rho_0 ) _0}{ 2 (k ' + h _0 ^{\vee}
)}. \label{lowest-confw} \eea

\vskip 5mm

We shall make the following assumptions:
\bea
&&
V_{\g_0} (\mu_i) \otimes V_{\g_0} (\mu_j) = V_{\g_0} (\mu_l) \oplus  \bigoplus_{r=1} ^{m_{i,j}} V_{\g_0} (\nu_{r,i,j}) \label{uvj-1} \\
&& L_{\g}(k,0)^{l} \ \mbox{does not contain} \
\widehat{\g_0}-\mbox{singular vector of weight} \
\widehat{\nu_{r,i,j}}, \label{uvj-2} \\
&& ( \mbox{where} \ l = i + j \ (\mod \  s+1 ) \  ), \nonumber
 \eea
for all $i,j \in \{1, \dots, s\}$ and $r=1, \dots, m_{i,j}.$

\begin{thm} \label{general}
Assume that conditions (\ref{uvj-1}) and (\ref{uvj-2}) hold. Then
$$L_{\g}(k,0) = L_{\g_0}(k',0) \oplus L_{\g_0} (k', \mu_1) \oplus \cdots \oplus L_{\g_0} (k',\mu_s). $$
\end{thm}
\begin{proof}
We have already seen that $\widetilde{L}_{\g_0} (k', 0)$ is
conformally embedded into  $L_{\g} (k,0)$. We also have that $L_{\g}
(k,0)$ contains $\hat{\g_0}$--submodules $\widetilde{L}_{\g_0} (k',
\mu_i)$ generated by top components  $V_{\g_0} (\mu_i),
(i=1,\dots,s)$. By using  assumptions (\ref{uvj-1}) and
(\ref{uvj-2}) and standard fusion rules arguments (similarly as in
\cite[Lemma 8.1]{AP}) one can conclude that if $$u \in V_{\g_0}
(\mu_i) \subset \widetilde{L}_{\g_0} (k', \mu_i), \quad   v\in
V_{\g_0} (\mu_j)  \subset \widetilde{L}_{\g_0} (k', \mu_j) $$ then
$$ u_n v \in \widetilde{L}_{\g_0} (k', \mu_{l}) \qquad (l = i + j \  ( \mod \ s+1) ). $$
(Here $\mu_0 = 0$). So $L_{\g} (k,0)$ contains a vertex subalgebra
isomorphic to
$$\widetilde{L}_{\g_0} (k', 0) \oplus  \widetilde{L}_{\g_0} (k', \mu_1) \oplus \cdots \oplus \widetilde{L}_{\g_0} (k', \mu_s). $$
Since every generator of $\g$ belongs to this subalgebra we get that
$$L_{\g} (k,0)  = \widetilde{L}_{\g_0} (k', 0) \oplus  \widetilde{L}_{\g_0} (k', \mu_1) \oplus \cdots \oplus \widetilde{L}_{\g_0} (k', \mu_s). $$
Therefore $L_{\g} (k,0)$ is $\Z_{s+1}$--graded, and each component
is a highest weight $\hat{\g_0}$--module. Simplicity of $L_{\g}
(k,0)$ gives that $\widetilde{L}_{\g_0} (k', 0)$ is a simple vertex
operator algebra and $\widetilde{L}_{\g_0} (k', \mu_i)$ is its
simple module ($i=1, \dots, s$) (cf. \cite{DM}). The proof follows.
\end{proof}

\section{Some conformal embeddings}
\label{sec-3} In this section we shall apply results from Sections
\ref{kriterij} and \ref{kriterij-poluprost} and obtain some
conformal embeddings at negative integer and rational levels. In
this way we shall present a new uniform proof of conformal
embeddings which will include embeddings from \cite{AP} and
\cite{P4}, and give some new conformal embeddings associated to
pairs of simple Lie algebras $(G_2,A_2)$, $(E_6,F_4)$ and
$(D_{\ell}, B_{\ell-1})$.

\begin{enumerate}
\item[{}] Table 1.
 ${\g}$, $\g_0$ are  simple Lie
algebras, $\g_0 < \g$.
$$
\begin{array}{| c | c | c|  c | c | }
\hline
{\g} & {\g}_0& \mbox{decomposition of} \ \g& k&
k'    \\ \hline
A_{2\ell-1} & C_{\ell} & \g_0 \oplus V_{\g_0} ( \omega_2) & -1 & -1 \\
\hline
D_{\ell} & B_{\ell-1} & \g_0 \oplus V_{\g_0} (\omega_1) & -\ell + 2 & -\ell + 2 \\
\hline
E_{6} & F_{4} & \g_0  \oplus V_{\g_0} (\omega_4) & -3 & -3  \\
\hline
G_{2} & A_{2} & \g_0 \oplus V_{\g_0} (\omega_1) \oplus V_{\g_0}(\omega_2)& -5/3 & -5/3 \\
\hline
B_{\ell} & D_{\ell} & \g_0 \oplus V_{\g_0} (\omega_1) & -\ell + 3/2 & -\ell +3/2 \\
\hline
\end{array}
$$
\end{enumerate}

In the first three cases $\g_0$ is the fixed point subalgebra of
$\g$ of a Dynkin diagram automorphism of order two  described in
\cite{K}, Section 8. (Note also that the pair $(A_{2 \ell -1},
C_{\ell})$ was studied in \cite{AP}).

When $\g$ is a simple Lie algebra of type $G_2$, then $\g_0$ is a
Lie subalgebra generated by long root vectors. The pair $(B_{\ell},
D_{\ell})$ was described in \cite{P4} (see also \cite{FRT}).

\begin{thm} \label{general-2}
For every pair $(\g, \g_0)$ of simple Lie algebras from Table 1, we
have that $L_{\g}(k,0)$ is a completely reducible
$L_{\g_0}(k',0)$--module:
$$ L_{\g}(k,0) = L_{\g_0}(k',0) \oplus L_{\g_0}(k', \mu_1) \oplus \cdots L_{\g_0}(k',\mu_s) $$
where
$$\g = \g_0 \oplus  V_{\g_0}( \mu_1) \oplus \cdots V_{\g_0}(\mu_s), $$
$k$ and $k'$ are as in Table 1. In particular, $L_{\g}(k,0)$ is a
$\Z_{s+1}$--graded extension of $L_{\g_0}(k',0)$.
\end{thm}
\begin{proof}
 The assertion  when $(\g,\g_0) = (A_{2\ell-1}, C_{\ell})$ was proved in   \cite{AP} by using similar method. We shall now consider other cases.

 By using decompositions from Table 1, we see that there is a conformal embedding of $\widetilde{L}_{\g_0}(k',0)$ into $L_{\g}(k,0)$.
 Next we need to check that all pairs $(\g,\g_0)$  satisfy conditions from Theorem \ref{general}. We use the tensor product decompositions from Lemma \ref{tensor-dec-gen} (see below).
Assume first that $(\g,\g_0) \ne (E_6, F_4)$. Then we directly see
that for every weight  $\nu_{i,j,m}$ appearing in tensor product
decomposition, singular vector of weight $\widehat{\nu_{i,j,m}}$
does not have integral conformal weight (see Lemma
\ref{lowest-conf-w} below). Therefore there are no singular vectors
of such conformal weights in $L_{\g}(k,0)$.

By using different methods we directly see in Lemma
\ref{observationEF} that such condition holds for the pair $(E_6,
F_4)$. The proof now follows from Theorem \ref{general}.
\end{proof}

\begin{lem} \label{tensor-dec-gen} We have the following tensor product decompositions:
\item[(1)] $V_{A_2} (\omega_1) \otimes V_{A_2} (\omega_1) = V_{A_2} (2 \omega_1) \oplus V_{A_2} (\omega_2)$,
\item[(2)] $V_{A_2} (\omega_2) \otimes V_{A_2} (\omega_2) = V_{A_2} (2 \omega_2) \oplus V_{A_2} (\omega_1)$,
\item[(3)] $V_{A_2} (\omega_1) \otimes V_{A_2} (\omega_2) = V_{A_2} ( \omega_1 + \omega_2) \oplus V_{A_2}
(0)$,
\item[(4)]  $V_{F_4} (\omega_4) \otimes V_{F_4} (\omega_4) =
V_{F_4} (2 \omega_4) \oplus V_{F_4} (\omega_3) \oplus V_{F_4}
(\omega_4) \oplus V_{F_4} (\omega_1) \oplus V_{F_4}(0)$,
 \item[(5)]$V_{B_{\ell -1}}(\omega _1) \otimes V_{B_{\ell -1}}(\omega _1) = V_{B_{\ell -1}}(2 \omega _1) \oplus
V_{B_{\ell -1}}( \omega _2) \oplus V_{B_{\ell -1}}(0) \quad (\ell
\geq 4)$,
\item[(6)] $V_{D_{\ell}}(\omega _1) \otimes V_{D_{\ell}}(\omega _1) = V_{D_{\ell}}(2 \omega _1) \oplus
V_{D_{\ell}}( \omega _2) \oplus V_{D_{\ell}}(0)$.
\end{lem}
Using relation (\ref{lowest-confw}) we obtain:
\begin{lem} \label{lowest-conf-w} We have:
\item[(1)] The lowest conformal weights of $A_2 ^{(1)}$--modules with highest weights $-\tfrac{11}{3} \Lambda_0 + 2
\Lambda_i$ ($i=1,2$) and $-\tfrac{11}{3} \Lambda_0 + \Lambda_1 +
\Lambda_2$ are $\frac{5}{2}$ and $\frac{9}{4}$, respectively.
\item[(2)] The lowest conformal weights of $B_{\ell -1} ^{(1)}$--modules with highest weights $- \ell \Lambda_0 + 2 \Lambda _1$ and $- \ell
\Lambda_0 + \Lambda _2$ are $2+\frac{1}{\ell-1}$ and
$2-\frac{1}{\ell-1}$, respectively.
\item[(3)] The lowest conformal weights of $D_{\ell} ^{(1)}$--modules with highest weights $(- \ell -\frac{1}{2}) \Lambda_0 + 2
\Lambda _1$ and $(- \ell -\frac{1}{2}) \Lambda_0 + \Lambda _2$ are
$2+\frac{2}{2 \ell-1}$ and $2-\frac{2}{2 \ell-1}$, respectively.
\end{lem}

\begin{lem} \label{observationEF}
As an  $F_4 ^{(1)}$--module, $L_{E_6} (-3,0)$ does not contain
singular vectors of weights $-7 \Lambda_0 + 2 \Lambda _4$, $-5
\Lambda_0 + \Lambda _1$ and $-7 \Lambda_0 + \Lambda _3$.
\end{lem}
\begin{proof}
We first note that $F_4 ^{(1)}$--modules with highest weights $-7
\Lambda_0 + 2 \Lambda _4$ and $-5 \Lambda_0 + \Lambda _1$ do not
have integral conformal weights. (The corresponding lowest conformal
weights are $\frac{13}{6}$ and $\frac{3}{2}$, respectively.)

The lowest conformal weight of module of highest weight $-7
\Lambda_0 + \Lambda _3$ for $F_4 ^{(1)}$ is $2$, but one can
directly check that there is no non-trivial singular vector of that
weight in $L_{E_{6}}(-3,0)$. We use the construction of the root
system of type $E_6$ from \cite{Bou}, \cite{H}. For a subset $S=
\{i_1, \ldots ,i_k \} \subseteq \{1,2,3,4,5 \}$ with odd number of
elements (i.e $k=1,3$ or $5$), denote by $e_{(i_1 \ldots i_k)}$ the
(suitably chosen) root vector associated to positive root
$$ \frac{1}{2}\left(\epsilon_8 - \epsilon_7 - \epsilon_6 + \sum
_{i=1}^{5} (-1)^{p(i)} \epsilon_i \right), $$
such that $p(i)=0$ for $i \in S$ and $p(i) =1$ for $i \notin S$. We
will use the symbol $f_{(i_1 \ldots i_k)}$ for the root vector
associated to corresponding negative root vector. Since
\begin{eqnarray*}
&& v_{E_{6}} = ( e_{(5)}(-1)e_{(12345)}(-1) +
e_{(125)}(-1)e_{(345)}(-1) \\
&& \qquad + e_{(135)}(-1)e_{(245)}(-1) + e_{(235)}(-1)e_{(145)}(-1)
) {\bf 1}
\end{eqnarray*}
is a singular vector for $E_6 ^{(1)}$ in $N_{E_{6}}(-3,0)$, one
concludes that $v_{E_{6}} \in N_{E_{6}} ^1 (-3,0)$. Furthermore,
$v_{E_{6}}$ is the unique (up to a scalar) singular vector of
conformal weight $2$ in $N_{E_{6}}(-3,0)$.

 Any vector of weight $-7 \Lambda_0 + \Lambda _3$ for $F_4 ^{(1)}$ in $L_{E_6}
 (-3,0)$ is a linear combination of vectors:
\begin{eqnarray*}
&& e_{(235)}(-1)e_{(234)}(-1) {\bf 1}, \ e_{(235)}(-1)e_{\epsilon_5
+ \epsilon_4}(-1) {\bf 1}, \ e_{(145)}(-1)e_{(234)}(-1) {\bf 1}, \\
&& e_{(145)}(-1)e_{\epsilon_5 + \epsilon_4} (-1) {\bf 1},
e_{(12345)}(-1)e_{(4)}(-1) {\bf 1}, \ e_{(12345)}(-1)e_{\epsilon_5 -
\epsilon_1}(-1) {\bf 1}, \\
&& e_{(245)}(-1)e_{(134)}(-1) {\bf 1},  e_{(245)}(-1)e_{\epsilon_5 +
\epsilon_3}(-1) {\bf 1}, \ e_{(345)}(-1)e_{(124)}(-1) {\bf 1}, \\
&& e_{(345)}(-1)e_{\epsilon_5 + \epsilon_2}(-1) {\bf 1}.
\end{eqnarray*}
The singularity assumption then directly implies that such a vector
is a linear combination of vectors $f_{(1)}(0)v_{E_{6}} $ and
$f_{\epsilon_5 - \epsilon_4}(0)v_{E_{6}} $, which is clearly zero in
$L_{E_{6}} (-3,0)$. The proof follows.
\end{proof}

\section{Conformal embedding of $L_{G_{2}}(-2, 0)$ into $L_{B_{3}}(-2, 0)$ and $L_{D_{4}}(-2, 0)$}

In this section we consider the conformal embedding associated to a
pair of simple Lie algebras $(D_4,G_2)$. It turns out that this case
is more complicated and one can not directly apply general results
from Section~\ref{kriterij-poluprost}. The main difference between
this case and the cases studied in Section~\ref{sec-3} is that
$L_{D_{4}}(-2, 0)$ contains a non-trivial singular vector for $G_{2}
^{(1)}$ of conformal weight $2$. In order to prove that
$L_{D_{4}}(-2, 0)$ is a completely reducible $L_{G_{2}}(-2,
0)$--module we shall need more precise analysis, which uses the
conformal embedding of $L_{B_3}(-2,0)$ in $L_{D_4}(-2,0)$ from
Section \ref{sec-3}.

Let $\g$ be the simple complex Lie algebra of type $D_4$. Then $\g$
has an order three automorphism $\theta$, induced from the Dynkin
diagram automorphism, such that the fixed point subalgebra $\g_0$ is
isomorphic to the simple Lie algebra of type $G_{2}$. We have the
decomposition of $\g$ into eigenspaces of $\theta$:
\begin{eqnarray} \label{decomp-simple-G2}
\g  = \g _0 \oplus V_{\g _0 }^{(1)}(\omega _1) \oplus
V_{\g_0}^{(2)}(\omega _1),
\end{eqnarray}
where $V_{\g _0 }^{(1)}(\omega _1)$ is generated by highest weight
vector $e_{\epsilon_1-\epsilon_4}+ \xi e_{\epsilon_1+\epsilon_4}+
\xi ^2 e_{\epsilon_2+\epsilon_3}$ and $V_{\g_0}^{(2)}(\omega _1)$ by
$e_{\epsilon_1-\epsilon_4}+ \xi e_{\epsilon_2+\epsilon_3}+ \xi ^2
e_{\epsilon_1+\epsilon_4}$, for suitably chosen root vectors. (Here
$\xi$ denotes the primitive third root of unity.)

Lie algebra $\g _0 $ is also a subalgebra of Lie algebra of type
$B_3$ considered in Section \ref{sec-3}, which we now denote by $\g
'$. We have
\begin{eqnarray} \label{decomp-simple-G2-druga}
\g ' = \g _0 \oplus V_{\g _0}^{(3)}(\omega _1),
\end{eqnarray}
where $V_{\g _0}^{(3)}(\omega _1)$ is generated by highest weight
vector $e_{\epsilon_1-\epsilon_4}+ e_{\epsilon_1+\epsilon_4}-2
e_{\epsilon_2+\epsilon_3}$.

By using (\ref{decomp-simple-G2-druga}), the decomposition of $\g$
from Section \ref{sec-3} and the fact that $V_{B_3}(\omega _1) \cong
V_{G_2}(\omega _1)$ as $\g _0$--module, one obtains another
decomposition of $\g$:
\begin{eqnarray}
\g  = \g _0 \oplus V_{\g _0 }^{(3)}(\omega _1) \oplus
V_{\g_0}^{(4)}(\omega _1),
\end{eqnarray}
where $V_{\g_0}^{(4)}(\omega _1)$ is generated by
$e_{\epsilon_1-\epsilon_4}- e_{\epsilon_1+\epsilon_4}$.

By using (\ref{decomp-simple-G2-druga}) and results from Section
\ref{kriterij} we see that $\widetilde{L}_{G_2} (-2,0)$ is
conformally embedded into $L_{B_3}(-2,0)$, which is conformally
embedded in $L_{D_4}(-2,0)$. Recall that Theorem \ref{general-2}
gives that the vertex operator algebra $L_{D_4}(-2,0)$ has an order
two automorphism $\sigma$ which defines the following decomposition:
\begin{eqnarray} \label{decompVOA-BD}
L_{D_4}(-2,0) = L_{B_3}(-2,0) \oplus L_{B_3}(-2,\omega_1).
\end{eqnarray}

Denote by $\widetilde{L}_{G_{2}} ^{(i)}(-2, \omega _1)$ the $G_{2}
^{(1)}$--submodule of $L_{D_4}(-2,0)$ generated by top component
$V_{\g_0}^{(i)}(\omega _1)$, for $i=1,2,3,4$. Furthermore, one can
directly check that vector
\begin{eqnarray*}
&& v= (e_{\epsilon_1 - \epsilon_3}(-1) e_{\epsilon_1 -
\epsilon_4}(-1) +e_{\epsilon_2 - \epsilon_4}(-1) e_{\epsilon_2 +
\epsilon_3}(-1) + e_{\epsilon_2 + \epsilon_4}(-1) e_{\epsilon_1 +
\epsilon_4}(-1)  \\
&& \quad - e_{\epsilon_1 - \epsilon_3}(-1) e_{\epsilon_1 +
\epsilon_4}(-1) - e_{\epsilon_2 - \epsilon_4}(-1) e_{\epsilon_1 -
\epsilon_4}(-1) - e_{\epsilon_2 + \epsilon_4}(-1) e_{\epsilon_2 +
\epsilon_3}(-1) ){\bf 1}
\end{eqnarray*}
is the unique (up to a scalar) non-trivial singular vector of weight
$-5\Lambda_0 + \Lambda _2$ for $G_{2} ^{(1)}$ in $L_{D_{4}}(-2, 0)$.
Denote by $\widetilde{L}_{G_{2}}(-2, \omega _2)$ the $G_{2}
^{(1)}$--submodule generated by $v$. We also have that $\theta
(v)=v$ and $\sigma (v)=-v$.  This implies that $v \in L_{B_{3}}(-2,
\omega _1)$.

Now we consider $L_{B_3}(-2,0)$ and $L_{B_3}(-2,\omega_1)$ as $G_{2}
^{(1)}$--modules. We use the following lemmas:

\begin{lem} \label{tensor-G2} We have the following tensor product decompositions:
\item[(1)]  $V_{G_{2}}(\omega _1) \otimes V_{G_{2}}(\omega _1) = V_{G_{2}}(2 \omega _1) \oplus
V_{G_{2}}( \omega _2) \oplus V_{G_{2}}(\omega _1) \oplus
V_{G_{2}}(0),$
\item[(2)] $V_{G_{2}}(\omega _1) \otimes V_{G_{2}}(\omega _2) = V_{G_{2}}(\omega _1 + \omega _2) \oplus
V_{G_{2}}(2 \omega _1) \oplus V_{G_{2}}(\omega _1).$
\end{lem}

\begin{lem} \label{observationGB}
As a  $G_{2} ^{(1)}$--module, $L_{D_4} (-2,0)$ does not contain
singular vectors of weights $-6 \Lambda_0 + 2 \Lambda _1$ and $-7
\Lambda_0 + \Lambda _1 + \Lambda _2$.
\end{lem}
\begin{proof}
The proof follows from the observation that $G_{2} ^{(1)}$--modules
of highest weights $-6 \Lambda_0 + 2 \Lambda _1$ and $-7 \Lambda_0 +
\Lambda _1 + \Lambda _2$ do not have integral conformal weights.
(The corresponding lowest conformal weights are $\frac{7}{3}$ and
$\frac{7}{2}$, respectively.)
\end{proof}
\begin{prop} \label{decompGB-1} We have:
\item[(1)] $L_{B_{3}}(-2,0) = \widetilde{L}_{G_{2}}(-2, 0) + \widetilde{L}_{G_{2}} ^{(3)}(-2, \omega _1).$
\item[(2)] $L_{B_{3}}(-2,\omega _1) = \widetilde{L}_{G_{2}} ^{(4)}(-2, \omega _1) + \widetilde{L}_{G_{2}}(-2, \omega _2).$
\end{prop}
\begin{proof} (1) Since $v$ is the unique (up to a scalar) singular vector of weight
$-5\Lambda_0 + \Lambda _2$ for $G_{2} ^{(1)}$ in $L_{D_{4}}(-2, 0)$,
it follows that there is no singular vector of that weight in
$L_{B_{3}}(-2,0)$. The tensor product decomposition from Lemma
\ref{tensor-G2} (1), Lemma \ref{observationGB} and standard fusion
rules arguments now imply that $\widetilde{L}_{G_{2}}(-2, 0) +
\widetilde{L}_{G_{2}} ^{(3)}(-2, \omega _1)$ is a vertex subalgebra
of $L_{B_{3}}(-2,0)$ which contains generators of $L_{B_{3}}(-2,0)$.
The claim (1) follows. Claim (1), Lemmas \ref{tensor-G2} and
\ref{observationGB} and standard fusion rules arguments now imply
that $\widetilde{L}_{G_{2}} ^{(4)}(-2, \omega _1) +
\widetilde{L}_{G_{2}}(-2, \omega _2)$ is an
$L_{B_{3}}(-2,0)$--submodule of the irreducible module
$L_{B_{3}}(-2,\omega _1)$. This implies claim (2).
\end{proof}

\begin{thm} \label{decompGD} We have:
$$L_{D_{4}}(-2,0) = L_{G_{2}}(-2, 0) \oplus L_{G_{2}}(-2, \omega _2)
\oplus L_{G_{2}} ^{(1)}(-2, \omega _1) \oplus L_{G_{2}} ^{(2)}(-2,
\omega _1).$$
\end{thm}
\begin{proof} First we notice that relation (\ref{decompVOA-BD}) and Proposition
\ref{decompGB-1} give
$$L_{D_{4}}(-2,0) = \widetilde{L}_{G_{2}}(-2, 0) + \widetilde{L}_{G_{2}}(-2, \omega _2)
+ \widetilde{L}_{G_{2}} ^{(3)}(-2, \omega _1) +
\widetilde{L}_{G_{2}} ^{(4)}(-2, \omega _1),$$
which also implies
$$L_{D_{4}}(-2,0) = \widetilde{L}_{G_{2}}(-2, 0) \oplus \widetilde{L}_{G_{2}}(-2, \omega _2)
\oplus \widetilde{L}_{G_{2}} ^{(1)}(-2, \omega _1) \oplus
\widetilde{L}_{G_{2}} ^{(2)}(-2, \omega _1).$$
By using the fact that ${\Z}_3$--orbifold components of a simple
vertex operator algebra are simple (cf. \cite{DM}), we get that
$$V=\widetilde{L}_{G_{2}}(-2, 0) \oplus \widetilde{L}_{G_{2}}(-2, \omega
_2)$$
is a simple vertex operator algebra, and that $\widetilde{L}_{G_{2}}
^{(1)}(-2, \omega _1)$ and $\widetilde{L}_{G_{2}} ^{(2)}(-2, \omega
_1)$ are its irreducible modules. Since the automorphism $\sigma$
 acts as $-1$ on $\widetilde{L}_{G_{2}}(-2,
\omega _2)$, it follows that $V$ is ${\Z}_2$--graded, so its graded
components are also simple. Theorem 3 from \cite{DM} now implies
that $\widetilde{L}_{G_{2}} ^{(1)}(-2, \omega _1)$ and
$\widetilde{L}_{G_{2}} ^{(2)}(-2, \omega _1)$ are inequivalent
$V$--modules, and Theorem 6.1 from the same paper then implies that
$\widetilde{L}_{G_{2}} ^{(1)}(-2, \omega _1)$ and
$\widetilde{L}_{G_{2}} ^{(2)}(-2, \omega _1)$ are irreducible as
$L_{G_{2}}(-2, 0)$--modules. The proof follows.
\end{proof}
It follows from the proof of Theorem \ref{decompGD} that
$\widetilde{L}_{G_{2}} ^{(3)}(-2, \omega _1)$ and
$\widetilde{L}_{G_{2}} ^{(4)}(-2, \omega _1)$ are also irreducible
$L_{G_{2}}(-2, 0)$--modules, so we obtain
\begin{coro} We have:
\item[(1)] $L_{B_{3}}(-2,0) = L_{G_{2}}(-2, 0) \oplus L_{G_{2}}(-2, \omega
_1)$,
\item[(2)]
$L_{B_{3}}(-2,\omega _1) = L_{G_{2}}(-2, \omega _1) \oplus
L_{G_{2}}(-2, \omega _2).$
\end{coro}

\end{document}